\documentclass{amsart}
\usepackage{mathtools,amssymb,tikz-cd}
\usepackage[hide links]{hyperref}

\theoremstyle{plain}
\newtheorem{theorem}{Theorem}
\newtheorem{lemma}[theorem]{Lemma}
\newtheorem{corollary}[theorem]{Corollary}
\newtheorem{proposition}[theorem]{Proposition}
\newtheorem{conjecture}[theorem]{Conjecture}

\theoremstyle{definition}
\newtheorem{definition}[theorem]{Definition}

\numberwithin{equation}{section}

\begin{document}

\title{KK-rigidity of simple nuclear C$^*$-algebras}
\author{Christopher Schafhauser}
\address{Department of Mathematics, University of Nebraska--Lincoln, Lincoln, NE, USA}
\email{cschafhauser2@unl.edu}
\date{\today}
\thanks{This work was partially supported by NSF grants DMS-2000129 and DMS-2400178.}

\begin{abstract}
	It is shown that if $A$ and $B$ are unital separable simple nuclear \mbox{$\mathcal Z$-stable} C$^*$-algebras and there is a unital embedding $A \rightarrow B$ which is invertible on $KK$-theory and traces, then $A \cong B$.  In particular, two unital separable simple nuclear $\mathcal Z$-stable C$^*$-algebras which either have real rank zero or unique trace  are isomorphic if and only if they are homotopy equivalent.
	It is further shown that two finite strongly self-absorbing C$^*$-algebras are isomorphic if and only if they are $KK$-equivalent in a unit-preserving way.
\end{abstract}

\maketitle

\setcounter{tocdepth}{1}
\tableofcontents

\section*{Introduction}
\renewcommand{\thetheorem}{\Alph{theorem}}

A conjecture of Elliott from the early 1990s (\cite{Elliott95}) states that unital separable simple nuclear C$^*$-algebras are classified up to isomorphism by $K$-theoretic invariants, in direct analogy with the Connes--Haagerup classification of separably acting injective factors (\cite{Connes76,Haagerup87}).  After several decades of work by the C$^*$-algebra community at large, a positive answer was reached: unital separable simple nuclear $\mathcal Z$-stable C$^*$-algebras satisfying the universal coefficient theorem (UCT) are classified by their Elliott invariant, consisting of a suitable combination of $K$-theory and traces (\cite{Kirchberg, Phillips00, GLN1, GLN2, EGLN, TWW, CETWW}---see also \cite{Gabe19, CGSTW}).

Without significantly enlarging the invariant, $\mathcal Z$-stability is necessary for such a classification result (\cite{Toms08}), but the role of the UCT is far more mysterious.  There is no known example of a separable nuclear C$^*$-algebra failing the UCT, and the question of whether one exists is perhaps the most important open problem in the theory of nuclear C$^*$-algebras, dating back to \cite{Rosenberg-Schochet87}.  The purpose of this paper is to examine the role of the UCT in the classification theorem.

In the infinite setting, the Kirchberg--Phillips theorem (\cite{Kirchberg,Phillips00}) states that two unital separable simple nuclear purely infinite C$^*$-algebras (i.e., unital Kirchberg algebras) $A$ and $B$ are isomorphic if and only if there is an invertible element in $KK(A, B)$ which preserves the unit on $K_0$.  In the same spirit, in the finite setting, one would hope for a classification result without the UCT using some combination of $KK$-theory and traces as the invariant.  The following result may be viewed as proof of concept.

\begin{theorem}\label{thm:main}
	If $A$ and $B$ are unital separable simple nuclear $\mathcal Z$-stable $C^*$-algebras and there is a unital embedding $A \rightarrow B$ which induces an invertible element in $KK$-theory and an invertible map on traces, then $A \cong B$.
\end{theorem}

The condition that an embedding $A \rightarrow B$ is invertible in $KK$-theory is, loosely speaking, saying that it is invertible up to a weak notion of homotopy.  In particular, if there is a homotopy equivalence $A \rightarrow B$ which induces a bijection on traces, then $A \cong B$.  If $A$ and $B$ are C$^*$-algebras for which projections separate traces, then any homotopy equivalence $A \rightarrow B$ will induce a bijection on traces, leading to the following homotopy rigidity result.  Note that, in particular, this holds if $A$ and $B$ either have real rank zero or unique trace.

\begin{corollary}\label{cor:main}
	If $A$ and $B$ are unital separable simple nuclear $\mathcal Z$-stable $C^*$-algebras for which projections separate traces, then $A \cong B$ if and only if $A$ and $B$ are homotopy equivalent.
\end{corollary}

For strongly self-absorbing C$^*$-algebras, a stronger result holds.  Recall that a unital separable infinite dimensional C$^*$-algebra $A$ is \emph{strongly self-absorbing} if $A \cong A \otimes A$ via an isomorphism approximately unitarily equivalent to the first factor embedding.  Such C$^*$-algebras are always simple, nuclear, and $\mathcal Z$-stable, and the finite ones have unique trace.  Infinite strongly self-absorbing C$^*$-algebras are Kirchberg algebras, and hence, as noted above, they are isomorphic if and only if they are $KK$-equivalent in a unit-preserving way.  It turns out that this also holds in the finite case.

\begin{theorem}\label{thm:main-ssa}
	If $A$ and $B$ are finite strongly self-absorbing $C^*$-algebras, then $A \cong B$ if and only if there is an invertible $\kappa \in KK(A, B)$ such that the induced map $\kappa_0 \colon K_0(A) \rightarrow K_0(B)$ sends $[1_A]$ to $[1_B]$.
\end{theorem}

Note that strong self-absorption is necessary in the previous theorem: the result does not extend to unital separable simple nuclear $\mathcal Z$-stable finite C$^*$-algebras, even if one assumes the algebras have a unique trace.  Indeed, there is extra information encoded in the real character $K_0(A) \rightarrow \mathbb R$ induced by the unique trace on such a C$^*$-algebra $A$.  This can be seen in the setting of unital simple AF algebras with unique trace.  For instance, take $K_0(A) = \mathbb Q + \mathbb Q \sqrt{2} \subseteq \mathbb R$ with the order inherited from $\mathbb R$ and unit $[1_A] = 1$, and take $K_0(B) = \mathbb Q \oplus \mathbb Q$ with the strict ordering from the first coordinate and unit $[1_B] = (1, 0)$.  It is easy to see that $A$ and $B$ are $KK$-equivalent in a unit-preserving way but are not isomorphic.

In general, any non-UCT classification result in the finite case must account for some interaction between $KK$-theory and traces, the same way that the Elliott invariant includes the natural pairing $T \times K_0 \rightarrow \mathbb R$.  The following is the most optimistic candidate for a compatibility condition.  Note that this is true under the UCT since the $K$-theory map $\kappa_* \colon K_*(A) \rightarrow K_*(B)$ induced by $\kappa$ and the trace map $\gamma$ form an isomorphism between the Elliott invariants of $A$ and $B$.\footnote{Conversely, if $A$ and $B$ satisfy the UCT and $(\kappa_*, \gamma)$ is an isomorphism of the Elliott invariants, then $\kappa_*$ lifts to an invertible $\kappa \in KK(A, B)$ by \cite[Corollary~7.5]{Rosenberg-Schochet87}, and the pair $(\kappa, \gamma)$ satisfies the hypotheses of Conjecture~\ref{conj}.}

\begin{conjecture}\label{conj}
	Suppose $A$ and $B$ are unital separable simple nuclear $\mathcal Z$-stable $C^*$-algebras.  If there is an invertible $\kappa \in KK(A, B)$ and an affine homeomorphism $\gamma \colon T(B) \rightarrow T(A)$ with $\kappa_0([1_A]) = [1_B]$ and $\langle \gamma(\tau), x \rangle = \langle \tau, \kappa_0(x) \rangle$ for all $\tau \in T(B)$ and $x \in K_0(A)$, then $A \cong B$.
\end{conjecture}

This conjecture may be too naive.  It may be that some stronger, and perhaps more mysterious, compatibility condition is needed between $KK$ and traces.  Any modification of the pairing map for which the conjecture is true would be most welcome.  Of course, whatever the ``correct'' compatibility condition is should be automatic for a pair $(\kappa, \gamma)$ induced by a $^*$-homomorphism.  This is, in some sense, why no compatibility condition is needed in Theorem~\ref{thm:main}, where it is assumed that one starts with a $^*$-homomorphism inducing  $\kappa$ and $\gamma$.

The growing interest in group actions on ``classifiable'' C$^*$-algebras makes such a result more pressing.  Even under the somewhat restrictive condition where an equivariant UCT may be possible (i.e., for actions in a suitable bootstrap class), there are typically no satisfactory formulas for equivariant $KK$-theory.  To my knowledge, the only known equivariant UCTs are for certain finite group actions (\cite{Kohler10,Meyer-Nadareishvili}) and for actions of compact connected Lie groups with torsion free fundamental group (\cite{Rosenberg-Schochet86}), and even in these cases, the formulas obtained are rather complicated.  For this reason, obtaining classification results in the finite setting which do not depend on the UCT seems crucial for further development of the equivariant classification theory.\footnote{For equivariant classification results in the infinite setting, see \cite{Gabe-Szabo}.}

\subsection*{Acknowledgments}
A significant portion of this work was completed during my stay at the Fields Institute in the fall of 2023 during the thematic program on ``Operator Algebras and Applications,'' and I thank the institute and the organizers of the program for facilitating such a stimulating research environment.  In addition, this project substantially benefited from my participation in the AIM SQuaRE program ``von Neumann algebraic techniques in the classification of C$^*$-algebras'' during the period 2019--2023 with Jos\'e Carri\'on, Kristin Courtney, Jamie Gabe, Aaron Tikuisis, and Stuart White; I thank all of them for countless conversations over the last several years and thank AIM for their support.  I also thank Elo{\'i}sa Grifo for a helpful discussion about triangulated categories, which played a role in an earlier attempt at the results in this paper.

\renewcommand{\thetheorem}{\arabic{theorem}}
\numberwithin{theorem}{section}

\section{Preliminaries}\label{sec:preliminaries}

Throughout this article, $A$ and $B$ will denote unital separable simple nuclear $\mathcal Z$-stable C$^*$-algebras.  Although some of these conditions can be relaxed in some of the intermediate results, they are crucial for the results stated in the introduction.  For the sake of clarity, I chose to keep the same hypotheses on $A$ and $B$ throughout instead of aiming for maximal generality.

A unital separable simple nuclear $\mathcal Z$-stable C$^*$-algebra $A$ is either purely infinite or stably finite (\cite[Theorem~4.1.10(ii)]{Rordam-Stormer}), and in the latter case, the tracial state space of $A$, $T(A)$, is non-empty (\cite[Corollary~5.12]{Haagerup14}).  In the infinite case, all of the results in this paper are standard consequences of the Kirchberg--Phillips theorem, so the reader may prefer to assume the C$^*$-algebras are finite throughout the paper.

The paper will freely use basic results from $KK$-theory.  A thorough treatment can be found in \cite{Blackadar98} or \cite{Jensen-Thomsen91}.  For a much more condensed treatment, highlighting the aspects of $KK$-theory most used in the classification theory, see \cite[Section~5.1]{CGSTW}.  At various points in the paper, it will be necessary to work with $KK$-groups $KK(D, E)$ when $E$ is not $\sigma$-unital---however, the first variable $D$ will always be separable.  In this case, one can define
\begin{equation}\label{eq:kk-sep}
	KK(D, E) = \varinjlim \, KK(D, E_0),
\end{equation}
where the limit is taken over all separable subalgebras $E_0 \subseteq E$.\footnote{When $E$ is $\sigma$-unital, the limit definition in \eqref{eq:kk-sep} agrees with the standard definition; see \cite[Proposition~B.6]{CGSTW}, for example.}  See \cite[Appendix~B]{CGSTW} for further discussion.  

For a unital C$^*$-algebra $D$, let $T(D)$ denote the set of tracial states on $D$, viewed as a weak$^*$-compact convex subset of $D^*$. If $D$ and $E$ are unital C$^*$-algebras and $\phi \colon D \rightarrow E$ is a unital $^*$-homomorphism, write $T(\phi) \colon T(E) \rightarrow T(D)$ for the continuous affine map $\tau \mapsto \tau \phi$.  Also, write $[1_D]$ for the class of the unit of $D$ in the group $K_0(D)$.  For an element $\kappa \in KK(D, E)$, the induced map $K_0(D) \rightarrow K_0(E)$ will be denoted by $\kappa_0$.

When $A$ and $B$ are separable, the group $KK(A, B)$ carries a natural (often non-Hausdorff) topology (\cite{Dadarlat05}), and the group $KL(A, B)$ can be defined as the quotient of $KK(A, B)$ by the closure of $\{0\}$.  The advantage of working with $KL$ over $KK$ is that two $^*$-homomorphisms $\phi, \psi \colon A \rightarrow B$ which are approximately unitarily equivalent may define different classes in $KK(A, B)$, but they will coincide in $KL(A, B)$; in fact, this is the reason $KL$-groups were initially introduced by R\o{}rdam in \cite{Rordam95} (where they were only defined when $A$ satisfies the UCT).

The usual constructions of the topology on $KK(A, B)$ depend on the separability of $B$.  To define $KL(A, B)$ when $B$ is non-separable, following  \cite[Definition~5.5]{CGSTW}, define
\begin{align}
	Z_{KK(A, B)} &= \left\{ (\mathrm{ev_\infty^B})_*(\kappa) : 
	\! \begin{array}{ll} \kappa \in KK(A, C(\mathbb N^\dag, B)) \text{ with} \\[2pt] (\mathrm{ev}^B_n)_*(\kappa) = 0 \text { for all } n \in \mathbb N \end{array}\! \right\}
	\subseteq KK(A, B),
\end{align}
where $\mathbb N^\dag = \mathbb N \cup \{\infty\}$ is the one-point compactification of $\mathbb N$ (with $\mathbb N$ viewed as a discrete space) and, for $n \in \mathbb N^\dag$, $\mathrm{ev}_n^B \colon C(\mathbb N^\dag, B) \rightarrow B$ denotes the evaluation map $f \mapsto f(n)$.  Then define $KL(A, B) = KK(A, B) / Z_{KK(A, B)}$.  When $B$ is separable, $Z_{KK(A, B)}$ coincides with the closure of $\{0\}$ by \cite[Theorem~3.5]{Dadarlat05}, and so the above definition of $KL(A, B)$ coincides with the usual one in the case when $B$ is separable.  As with $KK$-theory, any map $\kappa \in KL(A, B)$ will induce a map $\kappa_0 \colon K_0(A) \rightarrow K_0(B)$; note that if $\tilde\kappa \in KK(A, B)$ is a lift of $\kappa$, then $\tilde\kappa_0 = \kappa_0$, so there is no ambiguity caused by overusing the notation in this way.

The main tool in the paper is the trace-kernel extension, which first arose somewhat implicitly in the work of Winter (\cite{Winter10,Winter12}) and Matui and Sato (\cite{MatuiSato}) as a tool for proving $\mathcal Z$-stability and has since been heavily used in both the classification and regularity theories.  Suppose $A$ is a unital separable simple nuclear $\mathcal Z$-stable finite C$^*$-algebra.  As noted above, this implies $T(A)$ is non-empty.  Let $\ell^\infty(A)$ denote the C$^*$-algebra of bounded sequences in $A$ and let
\begin{equation}
	A^\infty = \ell^\infty(A) / \big\{ a \in \ell^\infty(A) : \lim_{n \rightarrow \infty} \sup_{\tau \in T(A)} \tau(a_n^*a_n) = 0 \big\}
\end{equation}
be the uniform tracial sequence algebra of $A$.  Also, let
\begin{equation}
	A_\infty = \ell^\infty(A) / \big\{ a \in \ell^\infty(A) : \lim_{n \rightarrow \infty} \|a_n\| = 0 \big\}
\end{equation}
denote the norm sequence algebra of $A$.  Then the trace-kernel extension of $A$ is the extension
\begin{equation}
\begin{tikzcd}
	0 \arrow{r} & J_A \arrow{r}{j_A} & A_\infty \arrow{r}{q_A} & A^\infty \arrow{r} & 0,
\end{tikzcd}
\end{equation}
where $q_A$ is defined on representing sequences by $(a_n)_{n=1}^\infty \mapsto (a_n)_{n=1}^\infty$, $J_A = \ker(q_A)$ is the trace-kernel ideal of $A$, and $j_A$ is the inclusion map.  

For future use, recall the central surjectivity result of \cite[Theorem~3.3]{Kirchberg-Rordam14}, stating that $q_A$ restricts to a surjective map $A_\infty \cap A' \rightarrow A^\infty \cap A'$.  The more general result below follows from \cite[Propositions~4.5(iii) and~4.6]{Kirchberg-Rordam14}.\footnote{These results are stated in \cite{Kirchberg-Rordam14} with ultrapowers in place of sequence algebras, but after replacing all ultralimits with sequential limits, the same proofs apply here.}

\begin{proposition}\label{prop:central-surj}
	With the notation above, if $S \subseteq A_\infty$ is a separable subset, then $q_A$ restricts to a surjective map $A_\infty \cap S' \rightarrow A^\infty \cap q_A(S)'$.
\end{proposition}

Given a function $r \colon \mathbb N \rightarrow \mathbb N$ with $r(n) \rightarrow \infty$ as $n \rightarrow \infty$, let $r^* \colon A_\infty \rightarrow A_\infty$ be the $^*$-homomorphism defined on representing sequences by $(a_n)_{n=1}^\infty \mapsto (a_{r(n)})_{n=1}^\infty$.  If $D$ is a separable C$^*$-algebra, $\psi_\infty \colon D \rightarrow A_\infty$ is a $^*$-homomorphism, and $r$ is a map as above,
call the composition $r^* \psi_\infty$ a reparameterization of $\psi_\infty$.  This will be used in the ``intertwining via reparameterization'' result from \cite[Theorem~4.3]{Gabe20}, which is modeled on the asymptotic version in \cite[Proposition~1.3.7]{Phillips00}, where it is attributed to Kirchberg.  Briefly, if $\psi_\infty$ is unitarily equivalent to all of its reparameterizations, then there is a $^*$-homomorphism $\psi \colon D \rightarrow A$ such that $\psi_\infty$ is unitarily equivalent to $\iota_A \psi$, where $\iota_A$ is the canonical inclusion of $A$ into $A_\infty$ as constant sequences.

\section{Classification of approximate embeddings}\label{sec:embeddings}

Let $A$ and $B$ be unital separable simple nuclear $\mathcal Z$-stable finite C$^*$-algebras.  The goal of this section is to describe all unital embeddings $A \rightarrow B_\infty$ in terms of traces and $KK$-theory.  The results of this section are all essentially from \cite{CETW} and \cite{CGSTW}.  

The classification is broken into three parts, collected in Theorems~\ref{thm:tracial},~\ref{thm:lifts1}, and~\ref{thm:lifts2}.  Although not quoted directly in this paper, it should be noted that the latter two results depend crucially on the Elliott--Kucerovsky characterization of (nuclearly) absorbing extensions and representations in terms of pure largeness (\cite{Elliott-Kucerovsky01}), along with a series of results in the regularity theory allowing one to access pure largeness from conditions such as $\mathcal Z$-stability (\cite{Kucerovsky-Ng06,Ortega-Perera-Rordam12}).

The first step is a classification of unital embeddings $A \rightarrow B^\infty$ due to Castillejos, Evington, Tikuisis, and White in \cite{CETW}, which is based on the notion of complemented partitions of unity (CPoU), developed in their collaboration with Winter in \cite{CETWW}.

\begin{theorem}[{\cite[Theorem~A]{CETW}}]\label{thm:tracial}
	Let $A$ and $B$ be unital separable simple nuclear $\mathcal Z$-stable finite $C^*$-algebras.
	\begin{enumerate}
		\item\label{thm:tracial:unique} If $\theta, \rho \colon A \rightarrow B^\infty$ are unital embeddings with $T(\theta) = T(\rho)$, then $\theta$ and $\rho$ are unitarily equivalent.
		\item\label{thm:tracial:exist} If $\gamma \colon T(B^\infty) \rightarrow T(A)$ is a continuous affine map, then there is a unital embedding $\theta \colon A \rightarrow B^\infty$ with $T(\theta) = \gamma$.
	\end{enumerate}
\end{theorem}

The second step of the classification is to decide when a given unital embedding $A \rightarrow B^\infty$ lifts along $q_B$ to a unital embedding $A \rightarrow B_\infty$.  The idea goes back to \cite{Schafhauser20a}, where the question is addressed when $B$ is the universal UHF algebra $\mathcal Q$ (and with ultrapowers in place of sequence algebras).  

\begin{theorem}[{\cite[Theorem~7.10(i)]{CGSTW}}]\label{thm:lifts1}
	Let $A$ and $B$ be unital separable simple nuclear $\mathcal Z$-stable finite $C^*$-algebras and let $\theta \colon A \rightarrow B^\infty$ be a unital embedding.  There is a unital embedding $\psi \colon A \rightarrow B_\infty$ such that $q_B \psi = \theta$ if and only if there is $\kappa \in KK(A, B_\infty)$ such that
	\begin{equation} 
		\kappa_0([1_A]) = [1_{B_\infty}] \in K_0(B_\infty) \quad \text{and} \quad (q_B)_*(\kappa) = [\theta] \in KK(A, B^\infty).
	\end{equation}
\end{theorem}

The final step is to describe all unital embeddings $A \rightarrow B_\infty$ with prescribed tracial data up to unitary equivalence, assuming at least one such embedding exists.  The result is essentially given in \cite{CGSTW}, but it will be convenient to set this up slightly differently.   For two unital embeddings $\phi, \psi \colon A \rightarrow B_\infty$, if $\phi$ and $\psi$ are unitarily equivalent, then $T(\phi) = T(\psi)$.  Conversely, when $T(\phi) = T(\psi)$, there is a secondary obstruction to unitary equivalence in $KL$-theory, as described below.  

Suppose $D$ and $E$ are C$^*$-algebras with $D$ separable, $I \subseteq E$ is an ideal, and  $\phi, \psi \colon D \rightarrow E$ are $^*$-homomorphisms whose pointwise difference $\phi - \psi$ has range contained in $I$---the pair $(\phi, \psi)$ is often called a $(D, I)$-Cuntz pair.  Then $\phi$ and $\psi$ have a formal difference $[\phi, \psi] \in KK(D, I)$, and hence also define an element $[\phi, \psi] \in KL(D, I)$.  This can be used to define what is called the Cuntz--Thomsen picture of $KK$-theory in \cite{CGSTW} (see \cite{Cuntz87,Thomsen90}), which essentially views $KK$-theory as the set of Cuntz pairs up to a suitable notion of homotopy.

\begin{definition}\label{def:secondary}
	Suppose $A$ and $B$ are unital separable simple nuclear $\mathcal Z$-stable finite C$^*$-algebras. If $\phi, \psi \colon A \rightarrow B_\infty$ are unital embeddings with $T(\phi) = T(\psi)$, define
	\begin{equation}
		\langle \phi, \psi \rangle = [\mathrm{ad}(u)\phi, \psi] \in KL(A, J_B),
	\end{equation}
	where $u \in B_\infty$ is a unitary with $q_B \mathrm{ad}(u) \phi = q_B \psi$.
\end{definition}

The following result implies that the secondary invariant $\langle \phi, \psi \rangle$ is Definition~\ref{def:secondary} is well-defined.

\begin{proposition}[{cf.\ \cite[Remark~7.5]{CGSTW}}]\label{prop:secondary}
	Suppose $A$ and $B$ are unital separable simple nuclear $\mathcal Z$-stable finite $C^*$-algebras and $\phi, \psi \colon A \rightarrow B_\infty$ are unital embeddings with $T(\phi) = T(\psi)$.
	\begin{enumerate}
		\item There is a unitary $u \in B_\infty$ such that $q_B \mathrm{ad}(u) \phi = q_B \psi$.
		\item The element $[\mathrm{ad}(u) \phi, \psi] \in KL(A, J_B)$ is independent of the choice of unitary $u$.\footnote{The same proof shows that $\langle \phi, \psi \rangle$ can be viewed as a well-defined element of $KK(A, J_B)$, but this will not be needed.}
	\end{enumerate} 
\end{proposition}

\begin{proof}
	To see the existence of $u$, note that Theorem~\ref{thm:tracial}\ref{thm:tracial:unique} implies there is a unitary $\bar u \in B^\infty$ with $\mathrm{ad}(\bar u) q_B \phi = q_B \psi$.  By \cite[Proposition~2.1]{CETW}, the unitary group of $B^\infty$ is path-connected, and hence there is a unitary $u \in B_\infty$ such that $q_B(u) = \bar u$.  Then $q_B \mathrm{ad}(u) \phi = q_B \psi$.
	
	Suppose now that $u, v \in B_\infty$ are unitaries with $q_B \mathrm{ad}(u)\phi = q_B \psi = q_B \mathrm{ad}(v) \phi$.  Then $q_B(v^*u) \in B^\infty \cap q_B(\phi(A))'$.  By another application of \cite[Proposition~2.1]{CETW}, the unitary group of $B^\infty \cap q_B(\phi(A))'$ is path-connected. Also, by Proposition~\ref{prop:central-surj}, $q_B$ restricts to a surjective $^*$-homomorphism $B_\infty \cap \phi(A)' \rightarrow B^\infty \cap q_B(\phi(A))'$.  Hence there is a unitary $w \in B_\infty \cap \phi(A)'$ with $q_B(w) = q_B(v^* u)$.  Then $vwu^* \in J_B + \mathbb C 1_{B_\infty}$ is a unitary, and so
	\begin{equation}
		[\mathrm{ad}(u) \phi, \psi] = [\mathrm{ad}(vwu^*) \mathrm{ad}(u) \phi, \psi] = [\mathrm{ad}(v) \phi, \psi],
	\end{equation}
	where the first equality follows from \cite[Proposition~5.4(ii)]{CGSTW}, for example, and the second equality uses that $\mathrm{ad}(w) \phi = \phi$ by the choice of $w$.
\end{proof}

The final piece of the classification of unital embeddings $A \rightarrow B_\infty$ is given below.  Note that with the well-definedness of the secondary invariant $\langle \phi, \psi \rangle$ established, the result is essentially the same as \cite[Theorem~7.10]{CGSTW}.

\begin{theorem}[{cf.\ \cite[Thoerem~7.10]{CGSTW}}]\label{thm:lifts2}
	Suppose $A$ and $B$ are unital separable simple nuclear $\mathcal Z$-stable finite $C^*$-algebras.
	\begin{enumerate}
		\item\label{thm:lifts2:unique} Let $\phi, \psi \colon A \rightarrow B_\infty$ be unital embeddings.  Then $\phi$ and $\psi$ are unitarily equivalent if and only if $T(\phi) = T(\psi)$ and $\langle \phi, \psi \rangle = 0$.
		\item\label{thm:lifts2:exist} If $\psi \colon A \rightarrow B_\infty$ is a unital embedding and $\kappa \in KL(A, J_B)$ such that $\kappa_0([1_A]) = 0$, then there is a unital embedding $\phi \colon A \rightarrow B_\infty$ such that $T(\phi) = T(\psi)$ and $\langle \phi, \psi \rangle = \kappa$.
	\end{enumerate}
\end{theorem}

\begin{proof}
	The forward direction of \ref{thm:lifts2:unique} is immediate using the well-definedness of $\langle \phi, \psi \rangle$ (Proposition~\ref{prop:secondary}).  Indeed, if $u \in B_\infty$ is a unitary with $\mathrm{ad}(u) \phi = \psi$, then
	\begin{equation}
		\langle \phi, \psi \rangle = [\mathrm{ad}(u) \phi, \psi] = [\psi, \psi] = 0.
	\end{equation}
	For the converse, suppose $\langle \phi, \psi \rangle = 0$ and choose a unitary $u \in B_\infty$ such that $q_B \mathrm{ad}(u)\phi = q_B \psi$, so 
	$\langle \phi, \psi \rangle = [\mathrm{ad}(u)\phi, \psi]$.  Then the maps $\mathrm{ad}(u) \phi$ and $\psi$ both lift $\theta = q_B\psi \colon A \rightarrow B^\infty$, and $\tau \theta \in T(A)$ is faithful for all $\tau \in T(B^\infty)$ by the simplicity of $A$.  Since $[\mathrm{ad}(u)\phi, \psi] = 0$ in $KL(A, J_B)$, \cite[Thoerem~7.10(iii)]{CGSTW} implies 
	$\mathrm{ad}(u)\phi$ and $\psi$ are unitarily equivalent.  In particular, $\phi$ and $\psi$ are unitarily equivalent.
	
	The proof of \ref{thm:lifts2:exist} will use the same deunitization trick used in the proof of \cite[Theorem~7.10]{CGSTW} (and also in \cite{Schafhauser20a, Schafhauser20b}).  For any C$^*$-algebra $D$, let $\iota_2 \colon D \rightarrow M_2(D)$ denote the inclusion $\iota_2(d) = d \oplus 0_D$.  Given another C$^*$-algebra $E$ and a $^*$-homomorphism $\rho \colon D \rightarrow E$, the induced $^*$-homomorphism $M_2(D) \rightarrow M_2(E)$ is also denoted by $\rho$.
	
	Let $\tilde\kappa \in KK(A, J_B)$ be a lift of $\kappa$.  Define $\theta = q_B \psi \colon A \rightarrow B^\infty$ and note that $\theta$ is a unital embedding since $A$ is simple.  Then the unitization of $\iota_2 \theta$ is faithful and hence full as $A$ is simple.  Also, since $q_B \iota_2 \psi = \iota_2 \theta$, applying \cite[Proposition~7.9]{CGSTW} with $\iota_2 \psi$ and $\iota_2 \theta$ in place of $\psi$ and $\theta$ shows that there is a $^*$-homomorphism $\phi' \colon A \rightarrow M_2(B_\infty)$ such that $q_B \phi' = \iota_2 \theta$ and $[\phi', \iota_2 \psi] = (\iota_2)_*(\tilde\kappa)$.  Then by \cite[Proposition~5.4(iv)]{CGSTW},
	\begin{equation}
		(j_B\iota_2)_*(\tilde\kappa) = [\phi'] -[\iota_2\psi] \in KK(A, M_2(B_\infty)),
	\end{equation}
	and since $\kappa_0([1_A]) = 0$, it follows that 
	\begin{equation}
		[\phi'(1_A)] = [\iota_2(\psi(1_A))] = [\iota_2(1_{B_\infty})] \in K_0(M_2(B_\infty)).
	\end{equation}
	The assumptions on $B$ imply that $B$ has stable rank one (\cite[Thoerem~6.7]{Rordam04}), and hence so does $B_\infty$.\footnote{By \cite[Lemma~19.2.2]{Loring97}, for every element $b \in B_\infty$, there is a unitary $u \in B_\infty$ with $b = u|b|$.  Then $b$ is the limit of the invertible elements $u(|b|+\epsilon 1_{B_\infty})$ as $\epsilon \rightarrow 0^+$.} So there is a unitary $u \in M_2(B_\infty)$ with $\phi'(1_A) = u \iota_2(1_{B_\infty})u^*$ (cf.\ \cite[Proposition~6.5.1]{Blackadar98}).
	
	Note that $q_B(u) \in M_2(B^\infty)$ commutes with $\iota_2(1_{B^\infty})$ since 
	\begin{equation}
		q_B( \phi'(1_A)) = \iota_2( \theta(1_A))= \iota_2(1_{B^\infty}) = q_B(\iota_2(1_{B_\infty})).
	\end{equation}
	Using that $M_2(B^\infty) \cap \{\iota_2(1_{B^\infty})\}'$ has path-connected unitary group by \cite[Proposition~2.1]{CETW} and that $q_B$ restricts to a surjective map 
	\begin{equation}
		M_2(B_\infty) \cap \{\iota_2(1_{B_\infty})\}' \rightarrow M_2(B^\infty) \cap \{\iota_2(1_{B^\infty})\}'
	\end{equation} by Proposition~\ref{prop:central-surj}, there is a unitary $v \in B_\infty$ commuting with $\iota_2(1_{B_\infty})$ such that $q_B(v) = q_B(u)$.  Then $w = uv^*$ is a unitary in  $M_2(J_B) + \mathbb C 1_{M_2(B_\infty)}$  satisfying $\phi'(1_A) = w \iota_2(1_{B_\infty})w^*$.
	
	There is a unital embedding $\phi \colon A \rightarrow B_\infty$ such that $\mathrm{ad}(w) \iota_2 \phi = \phi'$.  By construction, $q_B \iota_2 \phi = q_B \iota_2 \psi$, and hence $q_B \phi = q_B \psi$.  As noted in \cite[Proposition~6.3(i)]{CGSTW}, $T(q_B)$ is invertible, and it follows that $T(\phi) = T(\psi)$.  Also,
	\begin{equation}
		[\iota_2 \phi, \iota_2 \psi] = [\phi', \iota_2 \psi] = (\iota_2)_*(\tilde\kappa) \in KK(A, M_2(J_B)),
	\end{equation}
	using \cite[Proposition~5.4(ii)]{CGSTW} for the first equality.
	Then since $(\iota_2)_*$ is an isomorphism by the stability of $KK$-theory, $[\phi, \psi] = \tilde\kappa \in KK(A, J_B)$,
	and in particular, $\langle \phi, \psi \rangle = \kappa \in KL(A, J_B)$.
\end{proof}

The existence result in the previous theorem is optimal, as the following proposition shows.

\begin{proposition}\label{prop:vanishing-unit}
	Suppose $A$ and $B$ are unital separable simple nuclear $\mathcal Z$-stable finite $C^*$-algebras and $\phi, \psi \colon A \rightarrow B$ are unital embeddings with $T(\phi) = T(\psi)$.  Then the map $\langle \phi, \psi \rangle_0 \colon K_0(A) \rightarrow K_0(J_B)$ induced by $\langle \phi, \psi \rangle$ vanishes on $[1_A]$.
\end{proposition}

\begin{proof}
	After replacing $\phi$ with a unitary conjugate of $\phi$, assume $q_B \phi = q_B \psi$, so that $\langle \phi, \psi \rangle = [\phi, \psi]$.  As noted in \cite[Proposition~5.4(iv)]{CGSTW}, $(j_B)_*([\phi, \psi]) = [\phi] - [\psi]$, and since $\phi$ and $\psi$ are unital, this implies
	\begin{equation}
		K_0(j_B)(\langle \phi, \psi \rangle_0([1_A])) = K_0(\phi)([1_A]) - K_0(\psi)([1_A]) = 0.
	\end{equation}
	Also, a CPoU argument implies $K_1(B^\infty) = 0$ (see \cite[Proposition~6.7]{CGSTW}, which is obtained by applying \cite[Proposition~2.1]{CETW} to matrix algebras over $B^\infty$), and hence the six-term exact sequence in $K$-theory induced by the trace-kernel extension implies $K_0(j_B)$ is injective.  The result follows.
\end{proof}

The following result collects some  arithmetic properties of the secondary invariants $\langle \phi, \psi \rangle \in KL(A, J_B)$ for later use.

\begin{proposition}\label{prop:secondary-arithmetic}
	Let $A_0$, $A$, and $B$ be unital separable simple nuclear $\mathcal Z$-stable finite $C^*$-algebras.  Suppose  $\alpha \colon A_0 \rightarrow A$ is a unital embedding and \mbox{$\phi, \psi, \rho \colon A \rightarrow B$} are unital embeddings with $T(\phi) = T(\psi) = T(\rho)$.  Then
	\begin{enumerate}
		\item\label{chain-rule} $\langle \phi, \psi \rangle + \langle \psi, \rho \rangle = \langle \phi, \rho \rangle$,
		\item\label{minus} $- \langle \phi, \psi \rangle = \langle \psi, \phi \rangle$, and
		\item\label{compose} $\alpha^*(\langle \phi, \psi \rangle) = \langle \phi \alpha, \psi \alpha \rangle$.
	\end{enumerate}
\end{proposition}

\begin{proof}
	Note that \ref{minus} is immediate from \ref{chain-rule} and the fact that $\langle \psi, \psi \rangle = 0$, which in turn follows from the fact that $[\psi, \psi] = 0$ and the well-definedness of $\langle \psi, \psi \rangle$.  Also, \ref{compose} follows directly from the definition of $\langle\cdot,\cdot\rangle$, together with the functoriality of $KL$-theory in the first variable.  It remains to show \ref{chain-rule}.  Fix unitaries $u, v \in B_\infty$ such that $q_B \mathrm{ad}(u) \phi = q_B \psi$ and $q_B \mathrm{ad}(v)\psi = q_B \rho$.  Since $KL$-theory is invariant under unitary equivalence,
	\begin{equation}
		[\mathrm{ad}(u) \phi, \psi] = \mathrm{ad}(v)_*([\mathrm{ad}(u) \phi, \psi]) = [\mathrm{ad}(vu) \phi, \mathrm{ad}(v)\psi].
	\end{equation}
	Also,
	\begin{equation}
		[\mathrm{ad}(vu) \phi, \mathrm{ad}(v) \psi] + [\mathrm{ad}(v)\psi, \rho] = [\mathrm{ad}(vu) \phi, \rho].
	\end{equation}
	Combining these two equations proves \ref{chain-rule}.
\end{proof}

\section{Proof of Theorem~\ref{thm:main}}\label{sec:main}

The basic strategy behind the proof of Theorem~\ref{thm:main} (and essentially all abstract isomorphism results in C$^*$-algebra theory) is Elliott's intertwining argument.  To show that a given unital embedding $\phi \colon A \rightarrow B$ is approximately unitarily equivalent to an isomorphism, it is enough to construct a unital embedding $\psi \colon B \rightarrow A$ which is a two-sided inverse for $\phi$ up to approximate unitary equivalence.  In fact, by exploiting the symmetry between $A$ and $B$, it will be enough to construct a unital embedding $\psi \colon B \rightarrow A$ which is a left inverse for $\phi$ up to approximate unitary equivalence.  Indeed, repeating the argument with $\psi$ in place of $\phi$ will give a left inverse for $\psi$, which will necessarily be approximately unitarily equivalent to $\phi$ by elementary algebra.

Some care is needed in this argument, however.  Once a left inverse $\psi$ is constructed, one must know that $[\psi]$ and $T(\psi)$ are invertible in order to repeat the argument to construct a left inverse of $\psi$.  Since $T$ is invariant under approximate unitary equivalence, it will follow that $T(\phi) T(\psi) = T(\psi\phi) = \mathrm{id}_{T(A)}$, and hence the invertibility of $T(\phi)$ implies that of $T(\psi)$.  However, this argument does not work on $KK$ since $KK$ is not invariant under approximate unitary equivalence.  The salient point is that the approximate unitary invariance of $KL$ does imply $[\psi]$ is invertible in $KL(B, A)$.  Using that the secondary invariant appearing in Theorem~\ref{thm:lifts2} takes values in $KL$ instead of $KK$, this will be sufficient to construct a left inverse for $\psi$, making use of the fact that a right inverse $\phi$ is already known to exist.

The following lemma will be applied twice in the proof of Theorem~\ref{thm:main-formal}: once as written (when constructing a left inverse $\psi$ for $\phi$) and once after reserving the roles of $A$ and $B$ (when constructing a left inverse $\phi'$ for $\psi$).

\begin{lemma}\label{lem:main}
	Suppose $A$ and $B$ are unital separable simple nuclear $\mathcal Z$-stable finite $C^*$-algebras and $\phi \colon A \rightarrow B$ is a unital embedding such that $[\phi] \in KL(A, B)$ and $T(\phi) \colon T(B) \rightarrow T(A)$ are invertible.  If there exists a unital embedding $\psi'_\infty \colon B \rightarrow A_\infty$ such that $T(\psi_\infty' \phi) = T(\iota_A)$, then there exists a unital embedding $\psi \colon B \rightarrow A$ such that $\psi \phi$ is approximately unitarily equivalent to $\mathrm{id}_A$.
\end{lemma}

\begin{proof}
	Since $[\phi]$ is invertible in $KL(A, B)$, the natural transformation 
	\begin{equation}
		\phi^* \colon KL(B, \,\cdot\,) \rightarrow KL(A, \,\cdot\,)
	\end{equation} 
	is also invertible.  By Proposition~\ref{prop:vanishing-unit}, $\langle \iota_A, \psi_\infty' \phi \rangle \in KL(A, J_A)$ vanishes on the unit in $K_0$.  Since $\phi$ is unital, this also holds for $(\phi^*)^{-1}(\langle \iota_A, \psi_\infty' \phi \rangle)$ in $KL(B, J_A)$.  Apply  Theorem~\ref{thm:lifts2}\ref{thm:lifts2:exist} to obtain a unital embedding $\psi_\infty \colon B \rightarrow A_\infty$ such that $T(\psi_\infty) = T(\psi_\infty')$ and
\begin{equation}\label{eq:3.5}
	\langle \psi_\infty, \psi_\infty' \rangle = (\phi^*)^{-1}(\langle \iota_A, \psi_\infty' \phi \rangle) \in KL(B, J_A).
\end{equation}
Applying $\phi^*$ to both sides of \eqref{eq:3.5} yields
\begin{equation}
	\langle \psi_\infty \phi, \psi_\infty'\phi\rangle =  \phi^*(\langle \psi_\infty, \psi_\infty' \rangle) = \langle \iota_A, \psi_\infty'\phi\rangle,
\end{equation}
using Proposition~\ref{prop:secondary-arithmetic}\ref{compose} for the first equality.  Then Proposition~\ref{prop:secondary-arithmetic}\ref{chain-rule} and~\ref{minus} imply $\langle \psi_\infty \phi, \iota_A \rangle = 0$, and hence $\psi_\infty\phi$ is unitarily equivalent to $\iota_A$ by Theorem~\ref{thm:lifts2}\ref{thm:lifts2:unique}.  

Let $r \colon \mathbb N \rightarrow \mathbb N$ be a function with $r(n) \rightarrow \infty$ as $n \rightarrow \infty$.  Then $r^* \psi_\infty \phi$ is unitarily equivalent to $r^* \iota_A = \iota_A$, which, in turn, is unitarily equivalent to $\psi_\infty \phi$.  As $T(\phi)$ is invertible, this implies $T(r^*\psi_\infty) = T(\psi_\infty)$.  Also,
\begin{equation}\label{eq:3.7}
	\phi^*(\langle r^* \psi_\infty, \psi_\infty \rangle) = \langle r^* \psi_\infty \phi, \psi_\infty \phi\rangle = 0,
\end{equation}
where the first equality follows from Proposition~\ref{prop:secondary-arithmetic}\ref{compose} and the second equality follows from Theorem~\ref{thm:lifts2}\ref{thm:lifts2:unique} using that $r^*\psi_\infty \phi$ and $\psi_\infty \phi$ are unitarily equivalent.  Since $\phi^*$ is invertible, \eqref{eq:3.7} implies $\langle r^*\psi_\infty, \psi_\infty \rangle = 0$, and then another application of Theorem~\ref{thm:lifts2}\ref{thm:lifts2:unique} yields that $r^*\psi_\infty$ and $\psi_\infty$ are unitarily equivalent.  As $r$ was arbitrary, \cite[Theorem~4.3]{Gabe20} implies that there is a unital embedding $\psi \colon B \rightarrow A$ such that $\iota_A \psi$ is unitarily equivalent to $\psi_\infty$.

Note that $\iota_A \psi \phi$ is unitarily equivalent to $\psi_\infty \phi$ and hence also to $\iota_A$.  It follows that $\psi \phi$ and $\mathrm{id}_A$ are approximately unitarily equivalent as morphisms $A \rightarrow A$.
\end{proof}

The following result is a more precise formulation of Theorem~\ref{thm:main}.  Although the algebras in the statement are not assumed to be finite, the new part of the theorem is the finite case.  The infinite case follows from a direct application of the Kirchberg--Phillips theorem.

\begin{theorem}\label{thm:main-formal}
	If $A$ and $B$ are unital separable simple nuclear $\mathcal Z$-stable $C^*$-algebras and $\phi \colon A \rightarrow B$ is a unital embedding with $[\phi] \in KK(A, B)$ and $T(\phi) \colon T(B) \rightarrow T(A)$  invertible, then there is an isomorphism $A \rightarrow B$ approximately unitarily equivalent to $\phi$.
\end{theorem}

\begin{proof}
	Note that since $T(A) \cong T(B)$, either $A$ and $B$ are both finite or both infinite.  In the infinite case, $A$ and $B$ are purely infinite, and the result follows from the Kirchberg--Phillips theorem.  For the rest of the proof, assume $A$ and $B$ are finite. 
	
	The first step is to construct a unital embedding $\psi_\infty' \colon B \rightarrow A_\infty$ as in Lemma~\ref{lem:main}.  Define 
	\begin{equation}
		\gamma = T(\phi)^{-1} T(q_A \iota_A) \colon T(A^\infty) \rightarrow T(B)
	\end{equation}
	and note that $\gamma$ is continuous and affine.  By Theorem~\ref{thm:tracial}\ref{thm:tracial:exist}, there is a unital embedding $\theta \colon B \rightarrow A^\infty$ such that $T(\theta) = \gamma$.  Then $T(\theta\phi) = T(q_A \iota_A)$.  By Theorem~\ref{thm:tracial}\ref{thm:tracial:unique}, $\theta\phi$ and $q_A \iota_A$ are unitarily equivalent.  
	
	As $[\phi]$ is invertible in $KK(A, B)$, so is the natural transformation 
	\begin{equation}
		\phi^* \colon KK(B,\,\cdot\,) \rightarrow KK(A, \,\cdot\,). 
	\end{equation}
     Let $\kappa = (\phi^*)^{-1}([\iota_A]) \in KK(B, A_\infty)$.  Then $\kappa_0([1_B]) = [1_{A_\infty}] \in K_0(A_\infty)$ since both $\iota_A$ and $\phi$ are unital.  Also, 
	\begin{equation}
		\phi^*((q_A)_*(\kappa)) = [q_A \iota_A] = [\theta \phi] = \phi^*([\theta]) \in KK(A, A^\infty).
	\end{equation}
	Since $\phi^*$ is invertible, this implies $(q_A)_*(\kappa) = [\theta]$, and hence by Theorem~\ref{thm:lifts1}, there is a unital embedding $\psi_\infty' \colon B \rightarrow A_\infty$ such that $q_A \psi_\infty' = \theta$.  Since $q_A \iota_A$ is unitarily equivalent to $\theta \phi = q_A \psi_\infty' \phi$,
	\begin{equation}
		T(\iota_A)T(q_A) = T(q_A \iota_A) = T(q_A \psi_\infty' \phi) = T(\psi_\infty' \phi) T(q_A),
	\end{equation}
	and since $T(q_A) \colon T(A^\infty) \rightarrow T(A_\infty)$ is invertible (see \cite[Proposition~6.3(i)]{CGSTW}, for example),  $T(\iota_A) = T(\psi_\infty' \phi)$.
	
	By Lemma~\ref{lem:main}, there is a unital embedding $\psi \colon B \rightarrow A$ such that $\psi \phi$ is approximately unitarily equivalent to $\mathrm{id}_A$.  In particular, as $KL$ and $T$ are invariant under approximate unitary equivalence,
	\begin{equation} 
		[\psi] = [\phi]^{-1} \in KL(B, A) \quad \text{and} \quad T(\psi) = T(\phi)^{-1} \colon T(A) \rightarrow T(B)	
	\end{equation}
	are invertible.  Let $\phi_\infty' = \iota_B \phi \colon B \rightarrow A_\infty$ and note that 
	\begin{equation}
		T(\phi_\infty' \psi) = T(\iota_B \phi \psi) = T(\psi) T(\phi) T(\iota_B) = T(\iota_B).
	\end{equation} 
	Applying Lemma~\ref{lem:main} with the roles of $A$ and $B$ reversed and with $\psi$ and $\phi_\infty'$ in place of $\phi$ and $\psi_\infty'$, there is a unital embedding $\phi' \colon A \rightarrow B$ such that $\phi' \psi$ is approximately unitarily equivalent to $\mathrm{id}_B$.

	Note that $\phi' \psi \phi$ is approximately unitarily equivalent to both $\mathrm{id}_B \phi = \phi$ and $\phi' \mathrm{id}_A = \phi'$, and in particular, $\phi$ and $\phi'$ are approximately unitarily equivalent.  It follows that $\phi$ and $\psi$ are inverses up to approximate unitary equivalence.  Finally, Elliott's intertwining argument (see \cite[Corollary~2.3.4]{Rordam-Stormer}, for example), implies that $\phi$ is approximately unitarily equivalent to an isomorphism $A \rightarrow B$.
\end{proof}

The homotopy rigidity result in the introduction is now immediate.

\begin{proof}[Proof of Corollary~\ref{cor:main}]
	Suppose that $A$ and $B$ are unital separable simple nuclear $\mathcal Z$-stable C$^*$-algebras for which projections separate traces and let $\phi \colon A \rightarrow B$ be a homotopy equivalence with homotopy inverse $\psi \colon B \rightarrow A$.  Then $\phi$ is an embedding since $A$ is simple and $[\phi] \in KK(A, B)$ is invertible with inverse $[\psi] \in KK(B, A)$.  By Theorem~\ref{thm:main-formal}, it suffices to show that $\phi$ is unital and $T(\phi)$ is invertible.
		
	Since $\phi\psi$ is homotopic to $\mathrm{id}_B$, there is a continuous path of projections from $\phi(\psi(1_B))$ to $1_B$, and hence $\phi(\psi(1_B)) = 1_B$.  Therefore, 
	\begin{equation}
		1_B = \phi(\psi(1_B)) \leq \phi(1_A) \leq 1_B,
	\end{equation}
	and hence $\phi$ is unital.  Note that by symmetry, $\psi$ is also unital.
	
	It remains to show $T(\phi)$ is invertible.  Fix $\tau \in T(A)$.  If $p \in A$ is a projection, then since $\psi\phi$ is homotopic to $\mathrm{id}_A$, $\psi(\phi(p))$ is homotopic to $p$ through a path of projections in $A$.  In particular, $\psi(\phi(p))$ and $p$ are unitarily equivalent, and hence $\tau(\psi(\phi(p))) = \tau(p)$.  Since this holds for all projections $p \in A$ and projections in $A$ separate traces, $\tau \psi \phi = \tau$.  Therefore, $T(\phi)T(\psi) = \mathrm{id}_{T(A)}$.  By symmetry, $T(\psi) T(\phi) = \mathrm{id}_{T(B)}$.  So $T(\phi)$ is invertible with inverse $T(\psi)$.
\end{proof}

\section{The strongly self-absorbing case}\label{sec:ssa}

Strongly self-absorbing C$^*$-algebras were introduced by Toms and Winter in \cite{Toms-Winter07} (cf.\ \cite{Kirchberg04}).  A unital separable infinite dimensional C$^*$-algebra $A$ is strongly self-absorbing if the first factor embedding $\mathrm{id}_A \otimes 1_A \colon A \rightarrow A \otimes A$ is approximately unitarily equivalent to an isomorphism.  Any such C$^*$-algebra $A$ has an approximately inner half flip in the sense that the embeddings 
\begin{equation}
	\mathrm{id}_A \otimes 1_A, 1_A \otimes \mathrm{id}_A \colon A \rightarrow A \otimes A
\end{equation}
are approximately unitarily equivalent (\cite[Proposition~1.5]{Toms-Winter07}), and then $A$ is simple and nuclear by \cite[Lemma~3.10]{Kirchberg-Phillips00} (cf.\ \cite[Propositions~2.7 and~2.8]{Effros-Rosenberg78}).  Also, $A$ is either purely infinite or has unique trace, as is noted in \cite[Theorem~1.7]{Toms-Winter07}.  Finally, note that $A$ is $\mathcal Z$-stable by main result of \cite{Winter11}.

Strongly self-absorbing C$^*$-algebras are very rigid, which makes them a natural test case for general problems about simple nuclear C$^*$-algebras, a viewpoint promoted by Winter (see \cite[Section~4.7]{Winter17}, for example).  Part of what makes the classification of strongly self-absorbing C$^*$-algebras more accessible than the general case is that for any two strongly self-absorbing C$^*$-algebras $A$ and $B$, there is at most one unital embedding $A \rightarrow B_\infty$ up to unitary equivalence.  This reduces the classification problem to proving an approximate existence theorem, characterizing when there exists a unital embedding $A \rightarrow B_\infty$.  This reduction will be used in the form of the following well-known result.

\begin{proposition}\label{prop:ssa-embedding}
	If $A$ and $B$ are strongly self-absorbing $C^*$-algebras, the following statements are equivalent.
	\begin{enumerate}
		\item\label{absorb} $A \otimes B \cong B$.
		\item\label{embed} There is a unital embedding $A \rightarrow B$.
		\item\label{approx-embed} There is a unital embedding $A \rightarrow B_\infty$.
	\end{enumerate}
\end{proposition}

\begin{proof}
	The downward implications are clear.  For \ref{approx-embed}$\Rightarrow$\ref{absorb}, by the Choi--Effros lifting theorem (\cite[Theorem~3.10]{Choi-Effros}), any unital embedding $A \rightarrow B_\infty$ lifts to a sequence of unital completely positive maps $\phi_n \colon A \rightarrow B$, which necessarily satisfy
	\begin{equation}
		\lim_{n \rightarrow \infty} \|\phi_n(a_1a_2) - \phi_n(a_1) \phi_n(a_2)\| = 0, \quad a_1, a_2 \in A.
	\end{equation}
	Then $B$ is $A$-stable by \cite[Proposition~4.2]{Dadarlat-Winter09}.
\end{proof}

The main extra ingredient that allows for the classification of strongly self-absorbing C$^*$-algebras in Theorem~\ref{thm:main-ssa} is the following result of Dadarlat and Winter, which is a special case of the UCT satisfied by self-absorbing C$^*$-algebras.

\begin{theorem}[{\cite[Theorem~3.4]{Dadarlat-Winter09}}]\label{thm:ssa-uct}
	If $A$ is a strongly self-absorbing $C^*$-algebra and $E$ is a separably $A$-stable $C^*$-algebra in the sense of \cite[Definition~1.4]{Schafhauser20b}, then there is an isomorphism
	$KK(A, E) \rightarrow K_0(E) \colon \kappa \mapsto \kappa_0([1_A])$.
\end{theorem}

\begin{proof}
	Since $E$ is separably $A$-stable, the set of separable $A$-stable subalgebras of $E$ is cofinal in the set of separable subalgebras of $E$.	Therefore,
	\begin{equation}
		KK(A, E) = \varinjlim\, KK(A, E_0) \quad \text{and} \quad K_0(E) = \varinjlim\, K_0(E_0),
	\end{equation}
	with the limits taking over all separable $A$-stable subalgebras $E_0 \subseteq E$.  By the naturality of the map $\kappa \mapsto \kappa_0([1_A])$, it suffices to prove the result when $E$ is separable and $A$-stable.  This is the content of \cite[Theorem~3.4]{Dadarlat-Winter09}.
\end{proof}

The previous result will be applied when $E$ is the tracial sequence algebra of a finite strongly self-absorbing C$^*$-algebra.  The key observation is that since all finite strongly self-absorbing C$^*$-algebras have the same tracial sequence algebra, if $A$ and $B$ are finite strongly self-absorbing C$^*$-algebras, then $B^\infty$ is separably $A$-stable.

\begin{proposition}\label{prop:tracial-ultrapower}
	Suppose $A$ is a finite strongly self-absorbing $C^*$-algebra. Then $A^\infty \cong R^\infty$, where $R$ is the separably acting hyperfinite ${\rm II}_1$ factor.
\end{proposition}

\begin{proof}
	This is a standard consequence of Connes's theorem and Kaplansky's density theorem.  Let $\tau$ denote the unique trace on $A$ and note that $\pi_\tau(A)''$ is a separably acting finite von Neumann algebra, which is injective as $A$ is nuclear (\cite[Theorem~6.4(iv)$\Rightarrow$(v) and Theorem~6.2(i)$\Rightarrow$(ii)]{Effros-Lance77}).  Further, since $\tau$ is the unique trace on $A$, $\pi_\tau(A)''$ is a factor (see \cite[Thoerem~6.7.3]{Dixmier77}, for example).  Since $A$ is simple, and infinite dimensional, $\pi_\tau(A)''$ is also infinite dimensional.  Therefore, $\pi_\tau(A)''$ is a separably acting injective II$_1$ factor, and by Connes's theorem (\cite[Theorem~7.1]{Connes76}), it is isomorphic to $R$.
	
	Identify $\pi_\tau(A)''$ with $R$ and, abusing notation, let $\tau$ also denote the unique trace on $R$.  Define $\pi_\tau^\infty \colon A^\infty \rightarrow R^\infty$ on representing sequences by $(a_n)_{n=1}^\infty \mapsto (\pi_\tau(a_n))_{n=1}^\infty$.  Since $\pi_\tau$ is isometric with respect to the 2-norms, $\pi_\tau^\infty$ is faithful.  For any $b \in R^\infty$ represented by a bounded sequence $(b_n)_{n=1}^\infty \subseteq R$, by Kaplansky's density theorem,\footnote{This is using the standard fact that the the strong topology on the unit ball of $R$ is induced by the 2-norm $\|b\|_2 = \tau(b^*b)^{1/2}$. See \cite[Proposition~III.5.3]{Takesaki79}, for example.} there is a sequence $(a_n)_{n=1}^\infty \subseteq A$ such that $\| \pi_\tau(a_n) - b_n\|_2 < 1/n$ and $\|a_n\| \leq \|b_n\|$.  In particular, $(a_n)_{n=1}^\infty \subseteq A$ is a bounded sequence and hence defines an element $a \in A^\infty$.  By construction, $\pi_\tau^\infty(a) = b$, proving $\pi_\tau^\infty$ is surjective.
\end{proof}

The following result provides a one-sided version of Theorem~\ref{thm:main-ssa}, characterizing when one finite strongly self-absorbing C$^*$-algebra unitally embeds into another.

\begin{theorem}\label{thm:ssa-embedding}
	For finite strongly self-absorbing $C^*$-algebras $A$ and $B$,  there is a unital embedding $A \rightarrow B$ if and only if there is $\kappa \in KK(A, B)$ with $\kappa_0([1_A]) = [1_B]$.
\end{theorem}

\begin{proof}
	Applying Proposition~\ref{prop:tracial-ultrapower} to both $A$ and $B$ implies $A^\infty \cong B^\infty$.  In particular, there is a unital embedding $\theta \colon A \rightarrow B^\infty$.  Further, this isomorphism also implies $B^\infty$ is a quotient of $\ell^\infty(A)$, so $B^\infty$ is separably $A$-stable by \cite[Proposition~1.12]{Schafhauser20b}.  Note that
	\begin{equation} 	
		(q_B\iota_B)_*(\kappa)_0([1_A]) = K_0(q_B\iota_B)(\kappa_0([1_A])) = [1_{B^\infty}] = K_0(\theta)([1_A]). 
	\end{equation}
	Then Theorem~\ref{thm:ssa-uct} implies $(q_B\iota_B)_*(\kappa) = [\theta]$, and Theorem~\ref{thm:lifts1} implies there is a unital embedding $A \rightarrow B_\infty$.  The result follows from Proposition~\ref{prop:ssa-embedding}.
\end{proof}

Symmetrizing the previous theorem gives Theorem~\ref{thm:main-ssa}.

\begin{proof}[Proof of Theorem~\ref{thm:main-ssa}]
	Let $A$ and $B$ be finite strongly self-absorbing C$^*$-algebras and let $\kappa \in KK(A, B)$ be invertible with $\kappa_0([1_A]) = [1_B]$.  Applying Theorem~\ref{thm:ssa-embedding} to both $\kappa$ and $\kappa^{-1}$, there are unital embeddings $A \rightarrow B$ and $B \rightarrow A$.  Then two applications of Proposition~\ref{prop:ssa-embedding} yield $A \cong A \otimes B \cong B$.
\end{proof}

As a consequence of Theorem~\ref{thm:main-ssa}, both the classification problem and embedding problem for strongly self-absorbing C$^*$-algebras can be reduced to the infinite case.

\begin{corollary}\label{cor:ssa-embedding}
	Suppose $A$ and $B$ are finite strongly self-absorbing $C^*$-algebras.
	\begin{enumerate}
		\item\label{cor:ssa:emb} There is a unital embedding $A \rightarrow B$ if and only if there is a unital embedding $A \rightarrow B \otimes \mathcal O_\infty$.
		\item\label{cor:ssa:iso} $A \cong B$ if and only if $A \otimes \mathcal O_\infty \cong B \otimes \mathcal O_\infty$.
	\end{enumerate}
\end{corollary}

\begin{proof}
	In both parts, the forward direction is trivial.  To see the reverse direction in \ref{cor:ssa:emb}, suppose $\phi \colon A \rightarrow B \otimes \mathcal O_\infty$ is a unital embedding.  Recall that the unital embedding $\mathrm{id}_B \otimes 1_{\mathcal O_\infty} \colon B \rightarrow B \otimes \mathcal O_\infty$ is a $KK$-equivalence. Then $\phi^*([\mathrm{id}_B \otimes 1_{\mathcal O_\infty}]^{-1})$ is an element of $KK(A, B)$ preserving the unit on $K_0$, and hence Theorem~\ref{thm:ssa-embedding} implies there is a unital embedding $A \rightarrow B$.  Then the reverse direction of  \ref{cor:ssa:iso} follows by symmetrizing \ref{cor:ssa:emb} and applying Proposition~\ref{prop:ssa-embedding} to obtain $A \cong A \otimes B \cong B$.
\end{proof}

The reduction in  Corollary~\ref{cor:ssa-embedding} has a connection to the quasidiagonality problem for finite strongly self-absorbing C$^*$-algebras.
Let $\mathcal Q$ denote the universal UHF algebra.  In Winter's survey article \cite{Winter17}, several variations of the quasidiagonality problem are considered, including the following questions (\cite[Sections~4.9 and~4.10]{Winter17}).
\begin{description}
	\item[\hskip 1.3ex $\textbf{QDQ}_{\mathrm{finite\ s.s.a.}}$] Are all finite strongly self-absorbing C$^*$-algebras quasidiagonal?
	
	 \vskip .5ex
	 
	\item[\hskip 1.3ex $\mathbf{QDQ}_{\mathrm{infinite\ s.s.a.}}$] Is $\mathcal Q \otimes \mathcal O_\infty$ the second largest strongly self-absorbing C$^*$-algebra?
\end{description} 
Winter also asks in \cite[Question~7.2]{Winter17} if there are any formal implications between these problems.  The following result is the implication
\begin{equation}
	\textbf{QDQ}_{\mathrm{infinite\ s.s.a.}} \Longrightarrow \textbf{QDQ}_{\mathrm{finite\ s.s.a.}}. 
\end{equation}

\begin{corollary}
	If every strongly self-absorbing $C^*$-algebra not isomorphic to $\mathcal O_2$ unitally embeds into $\mathcal Q \otimes \mathcal O_\infty$, then every finite strongly self-absorbing $C^*$-algebra is quasidiagonal.
\end{corollary}

\begin{proof}
	If $A$ is a finite strongly self-absorbing C$^*$-algebra, then $A \not\cong \mathcal O_2$.  So by assumption, there is a unital embedding $A \rightarrow \mathcal Q \otimes \mathcal O_\infty$.  By Corollary~\ref{cor:ssa-embedding}\ref{cor:ssa:emb}, there is then a unital embedding $A \rightarrow \mathcal Q$, and as $\mathcal Q$ is quasidiagonal, so is $A$.	
\end{proof}

\end{document}